\documentclass[10pt,a4paper,reqno]{amsart}

\usepackage{amssymb}
\usepackage{enumitem}
\usepackage[ruled,noline,titlenumbered,linesnumbered,noend,norelsize]{algorithm2e}
\DontPrintSemicolon
\SetNlSty{footnotesize}{}{}
\IncMargin{0.5em}
\SetNlSkip{1em}
\SetKwIF{If}{ElseIf}{Else}{if}{then}{else if}{else}{endif}
\SetKw{Return}{return}
\SetKw{Pick}{pick}
\SetNlSkip{2em}
\newenvironment{algorithm-hbox}{\hbadness=10000\begin{algorithm}}{\end{algorithm}}

\usepackage[%draft, % do not do any hyper linking,
    pdftex, colorlinks=true, linkcolor=black, urlcolor=black, citecolor=black,
    pdfpagelabels, bookmarks, bookmarksnumbered, naturalnames=true,
    bookmarksopen, breaklinks, linktocpage,
		pdftitle={Multipass greedy coloring of simple uniform hypergraphs}, pdfdisplaydoctitle,
]{hyperref}

\theoremstyle{plain}
\newtheorem{theorem}{Theorem}

\newtheorem{lemma}[theorem]{Lemma}
\newtheorem{corollary}[theorem]{Corollary}

\newtheorem{prop}[theorem]{Proposition}

\numberwithin{equation}{section}

\newenvironment{enumeratei-continued}{\begin{enumerate}[label=\textup{(\roman*)}, noitemsep, topsep=1.5mm, labelindent=.8em, leftmargin=*, widest=., resume=my]}{\end{enumerate}}

\renewcommand{\epsilon}{\varepsilon}

\renewcommand{\leq}{\leqslant}
\renewcommand{\geq}{\geqslant}

\newcommand{\Nat}{\mathbb{N}}

\newcommand{\evS}{\mathcal{A}}
\newcommand{\vS}{\mathcal{X}}
\newcommand{\vars}{\mathrm{vbl}}
\newcommand{\prob}{\mathrm{Pr}}

\newcommand{\dg}{\mathcal{D}}

\newcommand{\ccc}{\mathcal{CC}}
\newcommand{\acc}{\mathcal{AC}}

\newcommand{\MGC}{MGC}
\newcommand{\MGCr}{MGC_r}

\newcommand{\dcc}{\mathcal{DC}}
\newcommand{\ncc}{\mathcal{NC}}
\newcommand{\idcc}{\mathcal{DI}}
\newcommand{\incc}{\mathcal{NI}}

\begin{document}
\title{Multipass greedy coloring of simple uniform hypergraphs}
\author{Jakub Kozik}
\address{Theoretical Computer Science Department, Faculty of Mathematics and Computer Science, Jagiellonian University, Krak\'{o}w, Poland}
\email{Jakub.Kozik@uj.edu.pl}

\thanks{Research of J.\ Kozik  was supported by Polish National
Science Center within grant 2011/01/D/ST1/04412.}

\begin{abstract}
Let $m^*(n)$ be the minimum number of edges in an $n$-uniform simple hypergraph that is not two colorable. 
We prove that $m^*(n)=\Omega( 4^n/\ln^2(n))$.
% improving previous bound of the order $4^n/ n^{o(1)}$.
% $m^*(n)=\Omega(4^n/ n^{\epsilon} )$, for any $\epsilon >0$.
%, derived by Kostochka  (2006) from the result of Szab\'{o} (1990) on maximum vertex degree in a simple uniform hypergraph that are not two colorable. 
Our result generalizes to $r$-coloring of $b$-simple uniform hypergraphs. For fixed $r$ and $b$ we prove that a maximum vertex degree in $b$-simple $n$-uniform hypergraph that is not $r$-colorable must be $\Omega(r^n /\ln(n))$. By trimming arguments it implies that every such graph has $\Omega((r^n /\ln(n))^{b+1/b})$ edges.
For any fixed $r \geq 2$ our techniques yield also a lower bound $\Omega(r^n/\ln(n))$ for van der Waerden numbers $W(n,r)$.
\end{abstract}

\maketitle

\section{Introduction and main results}
A hypergraph is a pair $(V,E)$, where $V$ is a set of vertices and $E$ is a family of subsets of $V$. 
Hypergraph is \emph{$n$-uniform} if all its edges have exactly $n$ elements. 
It  is \emph{simple} if any intersection of two edges has at most one element. 
A hypergraph is \emph{two colorable} if it is possible to assign to each vertex color blue or red in such a way that no edge is monochromatic. 
For historical reasons property of being two colorable is also named \emph{property B}. 
For $n\in \Nat$ let $m(n)$ be the smallest number of edges of an $n$-uniform hypergraph that is not two colorable. 
Parameter $m(n)$ has been introduced by  Erd{\H{o}}s and Hajnal in \cite{ErHaj1961}. 
The best known bounds on $m(n)$ are currently:
\[
	c \sqrt{n/\ln(n)}\;2^{n} \leq m(n) \leq (1+o(1)) \frac{e \ln{2}}{4} n^2 2^n.
\]
Upper bound was obtained by Erd{\H{o}}s in \cite{Er1964} and has not been improved since. 
The most recent improvement on the lower bound has been made by Radhakrishnan and Srinivasan in \cite{RS00}. 
Many results on similar extermal parameters of hypergraph coloring can be found in surveys by Kostochka \cite{Kos2006} and  Raigorodskii and Shabanov \cite{RaiSha2011}.

Erd{\H{o}}s and Lov\'asz in \cite{EL1975} analysed a variant of $m(n)$ denoted by $m^*(n)$ which is the minimum number of edges in a simple $n$-uniform hypergraph that is not two colorable. 
The authors proved asymptotic bounds  $m^*(n) = \Omega(4^n/n^3)$ and $m^*(n)=O(n^4 4^n)$ (this function grows much faster than $m(n)$). 
The lower bound of Erd{\H{o}}s and Lov\'asz was derived from the observation that a hypergraph, which is not two colorable, must contain a vertex with large degree. 
Let us define $D^*(n)$ as the maximum  number such that every simple $n$-uniform hypergraph with maximum vertex degree strictly smaller than $D^*(n)$ is two colorable. 
In \cite{Szabo1990}
Szab\'{o}
proved that for any $\epsilon>0$ and all large enough $n$, any simple $n$-uniform hypergraph with maximum vertex degree smaller than $n^{-\epsilon} 2^n$ is two colorable. 
It can be rephrased as $D^*(n) = n^{-\epsilon(n)} 2^n$ for some $\epsilon(n)= o(1)$ which was not specified explicitly.  
In \cite{Sha2012RSA} Shabanov proved that analogous bound is valid for  $\epsilon(n)=O(\sqrt{\ln\ln(n)/\ln(n)})$.
Using the method of Erd{\H{o}}s and Lov\'asz from \cite{EL1975} with a refinement observed by Kostochka \cite{Kos2006} these results implies the following bound 
\[
	m^*(n) = \Omega( 4^n/n^{2\epsilon(n)}).
\]

The first of our main results is an improvement of the lower bound on $D^*(n)$. 
\begin{theorem}
\label{thm:maxD}
Let $D^{*}(n)$ be the maximum number such that every simple $n$-uniform
hypergraph with maximum vertex degree at most $D^*(n)$ is two colorable. Then
\[
	D^*(n) = \Omega \left( \frac{2^n}{\ln(n)}  \right).
\]
\end{theorem}

We sketch below how an improved bound on $m^*(n)$ can be derived from the theorem above. We start with an observation by Erd\H{o}s and Lov\'{a}sz from \cite{EL1975}.
%By the observation of Erd\H{o}s and Lov\'{a}sz from \cite{EL1975} it implies a lower bound on the minimum number of edges in a simple $n$ uniform hypergraph which is not two colorable of the order $d^2/n$. 
Let $H$ be a simple $n$-uniform  hypergraph which is not two colorable.
%with maximum vertex degree $d<D*(n)$. 
We remove from every edge a vertex of this edge with maximum degree. 
The hypergraph becomes $(n-1)$-uniform, it is still simple and not two colorable. 
By Theorem \ref{thm:maxD} if $n$ is large enough then it contains a vertex $v$ of degree at least $d= \frac{c}{\ln(n)} 2^n$ for some positive constant $c$. Vertex $v$ has at least the same degree in $H$. 
For every edge $f$ adjacent to $v$, the vertex removed from $f$ has degree at least $d$. 
Since the hypergraph is simple there are at least $d$ such vertices. 
Altogether we have at least $d$ vertices with degrees at least $d$. 
Therefore the number of edges is at least $d^2/n$. Kostochka in \cite{Kos2006} observed that this bound can be improved as follows. 
Let us order these $d$ vertices linearly. 
Then the number of edges containing $i$-th vertex and some previous vertex is at most $i-1$ (because the hypergraph is simple). 
Therefore there are at least $\sum_{i=0}^{d-1} (d-i)= \binom{d+1}{2}$ distinct edges in the hypergraph. 
That gives the following corollary.

\begin{corollary}
\label{cor:m*n}
Let $m^*(n)$ be the minimum number of edges in a simple $n$-uniform hypergraph which is not two colorable. Then,
	\[
		m^*(n)= \Omega\left(\frac{4^n}{\ln(n)^2}\right).
	\]
\end{corollary}

The main result of the paper \cite{Szabo1990} by Szab\'{o} was a lower bound on van der Waerden numbers $W(n,2)$. 
Number $W(n,2)$ is the smallest integer such that if we partition set $\{1,2, \ldots ,W(n,2)\}$ into two classes, then at least one of the classes contains an arithmetic progression of length $n$. 
Szab\'{o} proved that for every $\epsilon>0$ and all large enough $n$ we have $W(n,2)\geq n^{-\epsilon} 2^n$. 
Our techniques also extends to this case and we show an improved lower bound which is also valid for larger number of colors.

\begin{theorem}
\label{thm:wn2}
\label{thm:Wnr}
Let $W(n,r)$ be the smallest integer such that if we partition the integers $\{1, 2, \ldots, W(n,r)\}$ into $r$ sets then at least one of them contains an arithmetic progression of length $n$. 
Then for any fixed $r$ we have
	\[
		W(n,r) = \Omega\left(\frac{r^n}{\ln(n)}\right).
	\]
\end{theorem}

Until recently the best lower bound on $W(n,2)$ valid for all large enough $n$ was the bound by Szab\'{o}. 
For prime $p$ it was proved by Berlekamp in \cite{Berlekamp1968} that $W(p+1,2)\geq p 2^p$. 
A lower bound stronger than ours, also valid for $r$ coloring, $W(n,r) = \Omega\left(\frac{r^n \ln\ln(n)}{\ln(n)}\right)$ has been recently announced in \cite{ShaVdW13}.
The best upper bound 
$W(n,2)\leq 2 \uparrow 2 \uparrow 2 \uparrow 2 \uparrow 2 \uparrow (n+9)$ has been proved by Gowers in \cite{Gowers2001}.

Hypergraph is called \emph{$b$-simple} if intersection of every two different edges contains at most $b$ elements. 
As a consequence in any $b$-simple hypergraph every set of vertices of size $b+1$ is contained in at most one edge. 
Kostochka and Kumbhat \cite{KosKum2009} define $f(n,r,b)$ as the smallest possible number of edges in an $n$-uniform $b$-simple hypergraph that is not $r$-colorable. 
For every positive $\epsilon$ and fixed $b$ and $r$ they proved that $f(n,r,b) = \Omega(\frac{r^{n(1+1/b)}}{n^\epsilon})$. 
The main ingredient of their proof was a result on maximum edge-degree of $n$-uniform $b$-simple hypergraphs that is not $t$-colorable. 
In order to improve their lower bound for $f(n,r,b)$ we derive a new bound on the maximum vertex degree in such graphs.

\begin{theorem}
\label{thm:Dbnr}
	Let $D^b(n,r)$ be the maximum number such that every $b$-simple $n$-uniform hypergraph with maximum vertex degree smaller than $D^b(n,r)$ is $r$-colorable. Then for every fixed $r \geq 2$ and $b$ we have
	\[
		D^b(n,r)= \Omega\left(\frac{r^n}{\ln(n)} \right).
	\]
\end{theorem}
We use simple trimming arguments, close to developments of Kostochka and Kumbhat from \cite{KosKum2009}, to derive from the theorem above the following corollary.
\begin{corollary}
\label{cor:fnrb}
Let  $f(n,r,b)$ be the smallest possible number of edges in an $n$-uniform $b$-simple hypergraph that is not $r$-colorable. Then for any fixed $b$ and $r$ we have
\[
	f(n,r,b) = \Omega\left( \left(\frac{r^n}{\ln(n)} \right)^{b+1/b} \right).
\]	
\end{corollary}
Let $m^*(n,r)$ the minimum number of edges in a simple hypergraph that is not $r$-colorable (i.e. $m^*(n,r)= f(n,r,1)$). 
Corollary above implies that for a fixed $r$ we have $m^*(n,r) = \Omega(r^{2n} / \ln(n)^2)$. It improves previous bound $m^*(n,r) \geq r^{2n} / n^{\epsilon(n)}$ for some $\epsilon(n)=O(\sqrt{\ln\ln(n)/\ln(n)})$, by Shabanov \cite{Sha2012RSA}.

Our proofs follow the strategy used already by Szab\'{o} in \cite{Szabo1990} -- 
for a fixed coloring procedure we analyse some specific configurations that make the procedure fail and use a variant of Local Lemma to prove that these configurations can be simultaneously avoided.
The main difference between our developments and those of Szab\'{o}  \cite{Szabo1990} and Kostochka and Kumbhat \cite{KosKum2009} is that we use different coloring procedure. 
Our procedure, called multipass greedy coloring ($\MGC$), is an extension of very natural greedy coloring first applied to hypergraph coloring problems by Pluh{\'a}r in \cite{Pluhar09}. 
Our extension and its analysis is close to the recent developments  by Cherkashin and Kozik \cite{CK2013} who used that procedure %analysed the performance of random greedy coloring procedure on uniform hypergraphs.
to improve lower bound on $m(n,r)$ -- the minimum number of edges in an $n$-uniform hypergraph that is not $r$-colorable. 
It turns out that obstacles that make $\MGC$ procedure fail are simpler than configurations considered in \cite{Szabo1990} and \cite{KosKum2009} which allows to derive better bounds.

\subsection{Organization of the paper and notation}
 Second section of the paper introduces special version of Local Lemma, which is used in the proofs of the main results. 
In the third section we define and discuss Multipass Greedy Coloring ($\MGC$) procedure that tries to color a hypergraph properly with  two colors. 
The procedure is parametrized by some function called birth time assignment.
In the fourth section we prove that for simple uniform hypergraph with suitable bound on maximum vertex degree it is possible to find a birth time assignment function which given on the input of $\MGC$ procedure produces a proper two coloring.
In the fifth section we generalize $\MGC$ procedure to the case of larger number of colors.
In Section \ref{sec:vdwr} we consider hypergraphs of arithmetic progressions and derive a lower bound on van der Waerden numbers $W(n,r)$. 
Section \ref{sec:bsimp} is devoted to the bounds on maximum vertex degree and number of edges in $n$-uniform $b$-simple hypergraphs that are not $r$-colorable.

Within the paper we assume that all vertices, edges etc. are linearly ordered in some fixed way (default ordering). 
The default ordering induces orderings on subsets of the vertices and edges (default induced orderings). 
Considered sets are always finite. 
An index of an element within a set is its position in the default induced ordering of the set. 
In this convention we say about index of a vertex within an edge, index of some edge among the edges containing specific vertex etc. 
We use $[m]$ to denote the set $\{1,\ldots,m\}$. 
For a function $b$ and subset of its domain $f$ let $b[f]$ denote the image of $f$ under $b$ (i.e.  $b[f]=\{y: \exists_{x\in f} y=b(x)\}$). 
%For a formal power series $F(z)$ of formal variable $z$ we denote by $[z^n]F(z)$ the coefficient of monomial $z^n$ within $F(z)$. 

\section{Local Lemma}
The proof of the lower bound on $m^*(n)$ from \cite{EL1975} contains the first application of the famous  Lov\'asz Local Lemma. 
We quote below its general version  as presented in \cite{AS08}.

\begin{lemma}
\label{lm:localG}
	Let $A_1, \ldots, A_n$ be events in arbitrary probability space. 
A directed graph $D=(V,E)$ on the set of vertices $V={1,\ldots, n}$ is called a dependency digraph for events $A_1, \ldots, A_n$ if for each $i$, $1\leq i \leq n$, the event $A_i$ is mutually independent of all the events $\{A_j: (i,j)\notin E\}$. 
Suppose that $D=(V,E)$ is a dependency digraph for the above events and suppose that there are real numbers $x_1, \ldots, x_n$ such that $0\leq x_i <1$ and $\prob(A_i)\leq x_i \prod_{j:(i,j)\in E}(1-x_j)$ for all $1\leq i\leq n$. 
Then
	\[
		\prob\left( \bigwedge_{i=1}^n  \overline{A_i}\right) \geq \prod_{i=1}^n (1-x_i).
	\]
In particular, with positive probability no event $A_i$ holds.
\end{lemma}

Szab\'{o} in \cite{Szabo1990} used a specific variant of Local Lemma derived from the general version by Beck in \cite{Beck80}. 
We use the following generalization of the  Beck's variant. 
The derivation from the general lemma follows exactly the proof presented in \cite{Beck80}. 
We quote this derivation for the sake of completeness.

\begin{lemma}
\label{lm:local}
	Let $\vS=(X_1, \ldots, X_m)$ be independent random variables in arbitrary probability space and let $\evS$ be a finite set of events determined by these variables. 
For $A\in \evS$ let $\vars (A)$  be the minimum set of variables that determines $A$. 
For $X\in \vS$ we define formal polynomial $w_X(z)$ in the following way:
\[
	w_X(z)= \sum_{A\in \evS: X\in \vars(A)} \prob(A) \; z^{|\vars(A)|}.
\]
Suppose that a polynomial $w(z)$ dominates all polynomials $w_X(z)$ for $X\in \vS$ i.e. for every  real $z_0 \geq 1$ we have $w(z_0) \geq w_x(z_0)$.
%for every $n\in \Nat$ and $X\in \vS$ we have $[z^n]w(z) \geq [z^n] w_X(z)$. 
If there exists $\tau_0\in (0,1)$ such that 
\[
	w(\frac{1}{1-\tau_0}) \leq \tau_0,
\]
then all events from $\evS$ can be simultaneously avoided.
\end{lemma}
\begin{proof}
	We are going to apply Lemma \ref{lm:localG}. 
Let us choose a graph $G=(V,E)$ such that $V=\evS$ and $(A_i,A_j)\in E$ if $\vars(A_i)\cap \vars(A_j) \neq \emptyset$. 
Clearly $G$  is a dependency graph. We choose $x_i=(1-\tau_0)^{-|\vars(A_i)|} \prob(A_i)$. Then
%\begin{align*}
%	(1-x_i)\prod_{j: \vars(A_i)\cap \vars(A_j) \neq \emptyset} x_j \geq 	
%	(1-x_i)\prod_{X \in \vars(A_i)} \prod_{j: X \in \vars(A_j)} x_j \geq \\
%	(1-x_i)\prod_{X \in \vars(A_i)} \left(1-\sum_{j: X \in \vars(A_j)} (1-x_j) \right) \geq
%	(1-x_i)\prod_{X \in \vars(A_i)} \left(1-w_X(\frac{1}{1-\tau_0}) \right) \geq\\
%	(1-x_i) \left(1-w(\frac{1}{1-\tau_0}) \right)^{|\vars(A_i)|}	 \geq
%	(1-x_i) (1-\tau_0)^{|\vars(A_i)|} = \prob(A_i)
%\end{align*}
%without align
\[
	x_i \prod_{j: (A_i,A_j)\in E} (1-x_j) \geq 	
	x_i\prod_{X \in \vars(A_i)} \prod_{j: X \in \vars(A_j)} (1-x_j)
\]
\[
	\geq 	x_i \prod_{X \in \vars(A_i)} \left(1-\sum_{j: X \in \vars(A_j)} x_j \right) \geq
	x_i\prod_{X \in \vars(A_i)} \left(1-w_X(\frac{1}{1-\tau_0}) \right)
\]
\[ 
	\geq
	x_i \left(1-w(\frac{1}{1-\tau_0}) \right)^{|\vars(A_i)|}	 \geq
	x_i (1-\tau_0)^{|\vars(A_i)|} = \prob(A_i).
\]	
It shows that assumptions of Lemma \ref{lm:localG} are satisfied so all events form $\evS$ can be simultaneously avoided.
\end{proof}

We say that $w_X(z)$ is a \emph{local polynomial for variable $X$} and $w(z)$ is \emph{local polynomial}. 
Usefulness of this form of Local Lemma lies in possibility of treating families of events independently. 
The proof of our main results are based on the following strategy used already in \cite{Szabo1990}. We work with $n$-uniform hypergraphs. 
After defining some set of bad events we partition it into finite number of types $\evS_1, \ldots, \evS_k$. 
Then, for each $i=1, \ldots, k$ and each variable $X\in \vS$ we define polynomial $w_X^{\evS_i}(z)= \sum_{A\in \evS_i: X\in \vars(A)} \prob(A) \; z^{|\vars(A)|}$ and choose some polynomial $w^{\evS_i}(z)$ that dominates all polynomials $w_X^{\evS_i}(z)$ for $X\in \vS$. 
Finally we always pick $\tau_0=1/n$. 
For notational convenience  we put $z_0= \frac{1}{1-\tau_0(n)}$. 
Then we show that $n\cdot w^{\evS_i}(z_0)$ is arbitrarily small for large enough $n$. 
We can choose $w(z)= \sum_{i=1}^k w^{\evS_i}(z)$ and then also $n \cdot w(z_0(n))$  is arbitrarily small for large $n$. 
Hence for all large enough $n$, the assumptions of the lemma are satisfied so all the bad event can be simultaneously avoided. 

Sometimes when we consider contribution to local polynomial $w_X(z)$ of some event $A$ dependent on $X$, it is not clear what is the exact cardinality of $\vars(A)$. 
That cardinality corresponds to the exponent of monomial $\prob(A) z^{|\vars(A)|}$ corresponding to that event. 
However, if we are looking for a polynomial that dominates $w_X(z)$ and  we know that $|\vars(A)| \leq K$, we can  instead use monomial $\prob(A) z^K$.

\section{Multipass greedy coloring}
\label{sec:MGC}
We present in this section a multipass greedy coloring procedure ($\MGC$) and derive some conditions which are sufficient for the algorithm to produce a proper two coloring. 
%$\MGC$ procedure is an extension of very natural greedy coloring procedure that has been used by Cherkashin and Kozik in \cite{CK2013} to analyse $m(n,r)$ -- the minimum number of edges in an $n$-uniform hypergraph that is not $r$-colorable. 

Let $(V, E)$ be a hypergraph. 
We fix some parameter $p \in (0, 1/2)$ and partition a unit interval $[0,1)$ into four subintervals
\[
	B=[0,\frac{1-p}{2}), \;\;\; P_B=[\frac{1-p}{2},\frac{1}{2}), \;\;\; 
	R=[\frac{1}{2}, 1- \frac{p}{2}), \;\;\; P_R = [1- \frac{p}{2},1).
\]
The input to the algorithm is a function $t$ which assigns numbers from $[0, 1)$ to the vertices of the hypergraph. 
It is convenient to consider points of the interval as points on a circle of unit circumference. 
Then intervals $B, R, P_B , P_R$ correspond to arcs of the circle. 
Any pair of points $(x, y) \in [0, 1)^2$ determines an arc of the circle $C(x, y)$ that starts at $x$ and continues clockwise along the circle until it hits $y$. 
\emph{Clockwise distance} between $x$ and $y$ is the length of $C(x, y)$.

For a vertex $v \in V$ the value $t(v)$ is called a \emph{birth time} of $v$. 
An edge $f \in E$ is \emph{degenerate} if $t[f ] \subset P_B \cup P_R$, it is \emph{easy} if $t[f] \cap B \neq \emptyset$ and $t[f] \cap R \neq \emptyset$. 
Any arc $c$ of the circle, except for the whole circle, induces a \emph{clockwise order} on the vertices with birth times in that arc (formally it can be defined as $v \preceq w \iff C(t[v], t[w]) \subset c)$.
For any edge $f \in E$ which is neither degenerate nor easy and for which $t[f ]\cap R = \emptyset$ (resp. $t[f] \cap B = \emptyset$), \emph{the first} and \emph{the last} vertex of $f$ is the first and the last vertex
of $f$ in the clockwise order induced by arc $P_R \cup B \cup P_B$ (resp. $P_B \cup  R \cup P_R$ ).

Multipass greedy coloring procedure ($\MGC$) is presented as Algorithm \ref{alg:MGC}. 
For convenience we demand that the input function $t$ is injective and makes no edge degenerate.
\begin{algorithm-hbox}[!ht]
\caption{Multipass Greedy Coloring ($\MGC$)}\label{alg:MGC}
 	\textbf{input:} injection $t:V \to [0,1)$ which makes no edge degenerate \\
 	\ForEach{$v\in V$ such that $t[v]\in B\cup P_B$}{ $c(v) \gets $ blue}
 	\ForEach{$v\in V$ such that $t[v]\in R\cup P_R$}{ $c(v) \gets $ red}
 	\textbf{let} $(b_1, \ldots, b_\alpha)$ be the vertices with birth times in $P_B$ ordered in such a way that $t(b_i)\leq t(b_{i+1})$ \\
 	\textbf{let} $(r_1, \ldots, r_\beta)$ be the vertices with birth times in $P_R$ ordered in such a way that $t(r_i)\leq t(r_{i+1})$ \\
 	\While{in the current coloring $c$ there exists a blue edge with last vertex in $P_B$ or a red edge with last vertex in $P_R$}{
 		\For{$i=1, \ldots, \alpha$}{
 			\If{$b_i$ is the last vertex of a blue edge}{ $c(b_i) \gets $ red}
 		}
 		\For{$i=1, \ldots, \beta$}{
 			\If{$r_i$ is the last vertex of a red edge}{ $c(r_i) \gets $ blue}
 		}

 	} 	
\Return $c$
\end{algorithm-hbox}	
Let us make a few observations about the procedure. 
The vertices with birth times in $B \cup R$ do not change their colors during the evaluation of the procedure. 
Therefore easy edges are colored properly while edges contained entirely in $B$ or in $R$ are monochromatic in the returned coloring function. 
Each vertex from $P_B$ is initially blue and it can change its color to red if at some point of the evaluation it is the last vertex of a blue edge (similarly for vertices from $P_R$). 
In particular the number of red vertices in $P_R$ and blue vertices in $P_B$ can only decrease. 
Moreover at least one of these values decreases in each iteration of the main loop (lines 8-14). 
Therefore the procedure always stops.

In order to analyse for which birth time assignments the procedure produces proper coloring we define the following notions.
A sequence of edges $(s_1 , \ldots , s_r )$ is a \emph{chain} if $|s_i \cap s_{i+1} | = 1$  for $i\in  [r - 1]$ and $(s_i \cap s_{i+1}) \cap (s_j \cap s_{j+1})=\emptyset$ for $i\neq j$. 
Every chain has \emph{corresponding sequence of vertices}  $(v_1 , \ldots , v_{r-1})$ such that $s_i \cap s_{i+1} = \{v_i\}$. By definition all vertices in the corresponding sequence are distinct.
We say that a chain contains a vertex if that vertex belongs to at least one edge of a chain.

For a birth time assignment $t$, a chain $(s_1,\ldots, s_r)$ with corresponding vertex sequence $(v_1,\ldots,v_{r-1})$ is \emph{alternating} if there are no easy or degenerate edges in the chain, 
%all vertices $v_1,\ldots,v_{r-1}$ are distinct, 
the last vertex of $s_{i+1}$ is the first vertex of $s_i$ and $v_i \in P_B \cup P_R$ for $i \in [r - 1]$. 
In fact if $v_i \in  P_B$ (resp. $v_i \in  P_R$) then $v_{i+1} \in P_R$ (resp. $v_{i+1} \in P_B$) since edges are non degenerate.
A chain $(s_1,\ldots,s_r)$ is \emph{conflicting} if it is alternating and the last vertex of $s_1$ belongs to $B$ or $R$. 
A conflicting chain is \emph{complete} if the first vertex of $s_r$ belongs to $B$ or $R$.

\begin{prop}
\label{prop:cc2}
If for injective $t : V \to [0, 1)$ there are no degenerate edges and
$\MGC$ procedure produces coloring which is not proper, then there exists a complete
conflicting chain w.r.t. $t$.
\end{prop}
\begin{proof}
Suppose that an edge $f$ is monochromatic in the coloring produced by the procedure. 
We assume that it is blue, the other case is symmetric.
We are going to construct a complete conflicting chain that starts with $f$. 
Note that $f$ can not have points in $R$ since these points are colored permanently red. 
Moreover the last vertex of $f$ can not lie in $P_B$ since the condition in the while loop (line 8) is false for the coloring returned by the procedure. 
Therefore $f \subset P_R \cup B$ and one element sequence $(f)$ is a conflicting chain. 
We are going to append edges to this sequence until the sequence becomes  complete. 
Additionally we keep two invariants as long as the sequence is not complete:	
\begin{enumerate}
	\item every first vertex of any edge in the sequence was recolored by the procedure,
	\item if $(v_1, \ldots, v_k)$ are the first vertices of edges of the sequence, then for $i=2, \ldots, k$ vertex $v_i$ has been recolored before vertex $v_{i-1}$. 
\end{enumerate}
If the first vertex of $f$ belongs to $B$ then the chain is complete.
In the opposite case the first vertex belongs to $P_R$ and it is blue so it must have been recolored. 
Therefore invariant (1) holds,  invariant (2) holds trivially. 
Suppose that we constructed a conflicting sequence $(f_1 , \ldots, f_k )$ which is not complete. 
It means that $f_k$ starts in $P_B$ or $P_R$ (wlog we assume that in $P_B$ ). 
Let us consider the first vertex of $f_k$.
It has been recolored, so at some point it was the last vertex of some edge $f$ which at that moment was blue. 
Let us observe that $f$ can not already belong to the sequence, because either it belongs to $B \cup P_B$ and there are no such edges in the sequence or $f$ starts in $P_R$ and its first vertex is blue which means that it has been recolored before the first vertex of $f_k$ . 
In particular it means that appending $f$ to the sequence constructed so far gives longer conflicting chain. 
Additionally whenever the extended chain is not complete both invariants are satisfied. 
Since the sequence can not be extended indefinitely, there must exist a complete conflicting chain.	
\end{proof}
%
%Proposition implies then if there are no degenerate edges and no complete conflicting chains w.r.t. birth-time assignment $t$, then the coloring produced by the procedure is proper. 
%Since any prefix of a conflicting chain is also a conflicting chain we have the following sufficient condition on $t$ under which the coloring produced by the procedure is proper.

\begin{corollary}
\label{cor:condT}
If there are no degenerate edges and no complete conflicting chains w.r.t. birth-time assignment $t$, then procedre $\MGC$ produces a proper coloring. 
\end{corollary}

\section{Simple uniform hypergraphs}
\label{sec:simp2}
We are going to apply Lemma \ref{lm:local} to prove that if an $n$-uniform hypergraph has appropriately bounded maximum vertex degree then there exists a birth time assignment function avoiding degenerate edges and complete conflicting chains. 
Chains of edges containing a fixed vertex play important role so let us start with a bound on their number. 
The bound in the following proposition holds in any $n$-uniform hypergraph and we are going to use it also for hypergraphs which are not simple.

\begin{prop}
\label{prop:disjC}
	In an $n$-uniform hypergraph with maximum vertex degree $d$, any
vertex belongs to at most $d k (nd)^{k-1}$ chains of length $k$.
\end{prop}
\begin{proof}
Let us fix vertex $v$ and let $s = (s_1,\ldots,s_k)$ be a chain containing $v$  with corresponding vertex sequence $(v_1,\ldots,v_{k-1})$. We associate to $s$ a signature (i.e. a tuple) with following entries:
\begin{enumerate}
\item index $i$ of the first edge of the chain that contains $v$ (i.e. smallest $i$ such that $v\in s_i$),
\item index of $s_i$ among the edges containing $v$,
\item sequence of $i - 1$ pairs of numbers $((a_2 , d_2 ),\ldots, (a_i , d_i ))$ where $a_j$ is the index of $v_{j-1}$ within $s_j$ and $d_j$ is the index of $s_{j-1}$ among the edges containing $v_{j-1}$,
\item sequence of $k  - i$ pairs of numbers $((a_{i+1},d_{i+1}),\ldots, (a_k, d_k))$ where $a_j$ is the index of $v_{j-1}$ within $s_{j-1}$ and $d_j$ is the index of $s_j$ among the edges containing $v_{j-1}$ .
\end{enumerate}

It is easy to check that different chains of length $k$ containing $v$ have different signatures. 
The first element of the signature is a number from $[k]$, the second from $[d]$, third and fourth form together a sequence of length $k - 1$ consisting of pairs of number – first from $[n]$ and second from $[d]$ (first element of the signature determines how the sequence should be split in two). 
Therefore there are no more than $kd(nd)^{k-1}$ signatures and the number of chains of length $k$ containing $v$ is not greater.
\end{proof}

Let $S$ be the set of all chains in considered hypergraph. 
We are going to apply Lemma \ref{lm:local} to prove that it is possible to  avoid events like $s$ is a complete conflicting chain (w.r.t. randomly chosen $t$) for any $s\in S$. 
The proposition above gives a bound on the number of such events that a fixed vertex can participate in. 
Probabilities of such events are more troublesome. 
In fact probability of an event that $s$ is conflicting depends on the sizes of intersections among the edges. 
A chain $(s_0,\ldots,s_{k-1})$ is \emph{disjoint} if $s_i\cap s_j = \emptyset$ for $|i-j| > 1$, it is an \emph{almost disjoint cycle} if $s_0\cap s_{k-1}\neq \emptyset$ and both $(s_0,\ldots,s_{k-2})$ and $(s_1,\ldots,s_{k-1})$ are disjoint chains.  
Clearly, every chain which is not disjoint contains a subsequence of consecutive elements which is an almost disjoint cycle of length at least 3. 
The definition of almost disjoint cycle allows the intersection $s_0 \cap s_{k-1}$ to be arbitrarily large and these are the only two edges of the cycle where this is allowed. 
In simple hypergraphs we have $|s_0 \cap s_{k-1}|=1$, therefore the chain is symmetric and any sequence $(s_{(j+1)_k},s_{(j+2)_k} \ldots,s_{(j+k)_k})$ (indices taken $\mod k$) is also a chain and an almost disjoint cycle. 
In a simple hypergraph with relatively large vertex degree, the number of almost disjoin cycles containing $v$ is much smaller than the bound on the number of chains containing $v$ from Proposition \ref{prop:disjC}. 

\begin{prop}
\label{prop:adisjC}
In a simple $n$-uniform hypergraph with maximum vertex degree $d$,
any vertex belongs to at most $k d (k-1) (nd)^{k-2} n^2$ almost disjoint cycles of length $k\geq 3$.
\end{prop}
\begin{proof}
Let us fix vertex $v$. Let $s = (s_0,\ldots, s_{k-1} )$ be an almost disjoint cycle containing $v$ with corresponding vertex sequence $(v_0, \ldots, v_{k-2})$. 
Let $v_{k-1}$ be the element of $s_0\cap s_{k-1}$. We assign the following signature to $s$:
\begin{enumerate}
	\item index $i$ of the first  edge of the chain that contains $v$,
	\item index of chain $(s_{(i)_k},s_{(j+2)_k} \ldots,s_{(i+k-2)_k})$ among the chains of length $k-1$ containing $v$,
	\item index of $v_{(i-1)_k}$ among the vertices of $s_i$,
	\item index of $v_{(i+k-2)_k}$ among the vertices of $s_{(i+k-2)_k}$.
\end{enumerate}
Once again different almost disjoint chains containing $v$ have different signatures. 
If two chains contain $v$ in edges on different position, they are distinguished by first entry. 
In the other case if the chains of $k-1$ edges starting from that position are different, the chains are distinguished by second entry. 
Finally the last edge is uniquely determined by vertices  $v_{(i-1)_k}$ and 
$v_{(i+k-2)_k}$ so if entries (3) and (4) are also the same for both chains then the chains must be equal. Counting is straightforward: 
first element of the signature is a number from $[k]$, second from $[d(k - 1)(nd)^{k-2} ]$ (by Proposition \ref{prop:disjC}), third  and fourth are numbers from $[n]$. 
Altogether it gives no more than $k d (k-1) (nd)^{k-2} n^2$ signatures and the number of almost disjoint chains of length $k$ containing $v$ is no greater.
\end{proof}

We are ready to prove the first main theorem.

\begin{proof}[Proof of Theorem \ref{thm:maxD}]
Let $(V, E)$ be an $n$-uniform simple hypergraph with maximum vertex degree $d$. It
is sufficient to prove that there exists an injective $t : V \to [0, 1)$ such that:
\begin{enumerate}
	\item there are no degenerate edges,
	\item and there are no complete conflicting disjoint chains,
	\item and there are no almost disjoint cycles which are alternating (as chains).
\end{enumerate}
Since every conflicting chain which is not disjoint contains an alternating almost
disjoint cycle, the conditions above imply the conditions of Corollary \ref{cor:condT}. Hence procedure $\MGC$ produces a proper two coloring when such $t$ is given on the input. 
We are going to apply Lemma \ref{lm:local}. 
For every vertex $v \in V$ we choose a birth time $t(v)$ uniformly at random from $[0, 1)$ independently from all other choices (so $\vS$ from the lemma is the sequence of these random variables). 
The set of events $\evS$ consists of three kinds of events:
\begin{enumerate}
	\item For every $f\in E$ let $\dg_f$ be the event that $f$ is degenerate.
	\item For every disjoint chain $s$ let $\ccc_s$ be the event that $s$ is  complete conflicting.
	\item For every almost disjoint cycle $s$ let $\acc_s$ be the event that $s$ is alternating.
\end{enumerate}
Each of these events is completely determined by birth times of vertices from corresponding chain or edge. 
We analyse contribution to the local polynomial for each kind of events. 
Let $S$ be the set of chains of edges, $S_d \subset S$  the set of disjoint chains and $S_a \subset S$ the set of almost disjoint cycles. 
Let $K$ be the maximum length of a disjoint chain in the hypergraph. 
In fact, any larger constant is also sufficient for our needs. 
\\
\textbf{Events $\dg$.} For every $f \in E$ we have:
\[
	\prob( \dg_f) = p^n.
\]
Each vertex $v \in V$ belongs to at most $d$ edges which means that a random variable $t(v)$ belongs to $\vars(f)$ for at most $d$ edges.
Therefore the contribution of the events of type $\dg$ to the local polynomial of $w_{t(v)}(z)$ is at most:
\[
	w_\dg (z) = p^n dz^n .
\]
\textbf{Events $\ccc$.}  
Let $(s_1,\ldots, s_k ) \in S_d$ and let $(v_1,\ldots, v_{k-1})$ be the corresponding sequence of vertices. 
For an event $\ccc_s$ to happen it is necessary that:
\begin{enumerate}
\item the sequence is alternating,
\item $t[s_1] \subset P_R \cup B$ or $t[s_1]\subset P_B \cup R$,
\item $t[s_k] \subset B \cup P_B$ or $t[s_k] \subset R \cup P_R$.
\end{enumerate}
Vertices $v_1 , v_2 ,\ldots, v_{k-1}$ must have birth times from alternating intervals $P_B$ and $P_R$. 
The probability of that is $2(p/2)^{k-1}$. 
Let us fix birth times for these vertices. 
For $i\in \{2,\ldots, k - 1\}$ all vertices of edge $s_i$ must have birth times in the arc $C(t(v_i ), t(v_{i-1}))$, let $\delta_i$ be the length of that arc (i.e. clockwise distance from $t(v_i)$ to $t(v_{i-1})$). 
The vertices of the first edge must all have birth times in $P_R \cup B$ or in
$P_B \cup R$ hence in the specific arc of length $\delta_1\leq 1/2$ 
(precisely from $t(v_1)$ to $(1-p)/2$ or $1-p/2$).
Similarly the vertices of the last edge must have birth times from the arc of  length $\delta_k \leq 1/2$ from $0$ or $1/2$ to $t(v_{k-1})$. 
Then the probability that all the vertices of the chain, except for $v_1,\ldots, v_{k-1}$, have birth times in corresponding arcs is $\delta_1 \delta_k \left( \prod_{i=1, \ldots, k} \delta_i \right)^{n-2}$. 
We have $\sum_{i=1, \ldots, k} \delta_i = (k-p)/2$ so the product is at most $(\frac{1-p/k}{2})^{k}$. Hence
\[
	\prob(\ccc_s) \leq 2 (p/2)^{k-1} \left(\frac{1-p/k}{2}\right)^{k(n-2)}.
\]
For a fixed $v\in V$ the number of sequences $s \in S_d$ of length $k$ that contains vertex $v$ is not greater than the number of chains of length $k$ that contain $v$, which is at most $dk(nd)^{k-1}$ by Proposition \ref{prop:disjC}. 
Every disjoint chain $s$ of length $k$ contains exactly $(n-1)k+1$ vertices. Accordingly the exponent at $z$ in the monomial corresponding to event $\ccc_s$ should be $(n-1)k+1$. 
To keep formulas simple we use upper bound $nk$.  
The contribution of the events of type $\ccc$ to the local polynomial is
at most:
\begin{align}
\label{eq:wcc}
	w_{\ccc}(z)& = \sum_{k=1}^{K} 
	2 (p/2)^{k-1} \left(\frac{1-p/k}{2}\right)^{k(n-2)}
	(dk(nd)^{k-1}) z^{kn} \nonumber \\
	& = \frac{1}{pn} \sum_{k=1}^{K} 4 k 
	 \left(\frac{dpn}{2^{n-1}}\right)^{k} (1-p/k)^{k(n-2)} z^{kn}.
\end{align}
\textbf{Events $\acc$.}  
Let $s=(s_1, \ldots, s_k)$ be an almost disjoint cycle  with corresponding vertex sequence $(v_1, \ldots, v_{k-1})$ that can occur as alternating. Let $v_k$ be such that $s_1 \cap s_k =\{v_k\}$. 
Just like in the previous case, vertices $v_1 , v_2 ,\ldots, v_{k-1}$ must have birth times from alternating intervals $P_B$ and $P_R$. 
Let us fix birth times of these vertices and for $i\in \{2,\ldots, k-1\}$ let $\delta_i$ be the clockwise distance from $t(v_i)$ to $t(v_{i-1})$. 
Additionally we put $\delta_1$ to be the clockwise distance from $t(v_1)$ to $1/2$ , if $t(v_1)\in P_R$ and from $t(v_1)$ to $1$ when  $t(v_1)\in P_B$. 
Analogously we put $\delta_k$ to be the clockwise distance from $(1-p)/2$ to $t(v_{k-1})$, if $t(v_{k-1})\in P_R$ and from $1-p/2$ to $t(v_{k-1})$ when  $t(v_{k-1})\in P_B$. 
Then conditional probability that vertices of the chain, except for $v_1, \ldots, v_{k-1}$, have birth times in corresponding arcs is $\delta_1 \left( \prod_{i=1,\ldots, k} \delta_i  \right)^{(n-2)}$.
 This time we have $\sum_{i=1,\ldots, k} \delta_i= (k+p)/2$ so the product is at most $(\frac{1+p/k}{2})^k$. Hence
\[
	\prob(\acc_s)  \leq 2 (p/2)^{k-1} \left(\frac{1+p/k}{2}\right)^{k(n-2)}.
\]
By Proposition \ref{prop:adisjC} the number of almost disjoint cycles of length $k$ containing vertex $v$ is smaller than $d^{k-1} n^k k^2$. Again we increase exponents at $z$ to make formulas simpler.
The contribution of the events of type $\acc$ to the local polynomial is
at most
\begin{align}
	w_{\acc}(z)& = \sum_{k=3}^{K} 
	2 (p/2)^{k-1} \left(\frac{1+p/k}{2}\right)^{k(n-2)} 
	(d^{k-1} n^k k^2) z^{kn} \nonumber \\
	&= \frac{4 n}{2^{n-1}} \sum_{k=3}^{K} k^2 \left(\frac{dpn}{2^{n-1}} \right)^{k-1} (1+p/k)^{k(n-2)} z^{kn}.
\end{align}

By the above considerations, polynomial
$
	w(z)= w_\dg(z)+w_\ccc(z)+w_\acc(z)
$
dominates all local polynomials $w_{t(v)}$ for $v\in V$. 
To apply Lemma \ref{lm:local} it is sufficient to find $\tau_0$ for which $w(\frac{1}{1-\tau_0}) \leq \tau_0$. 
Then we would know that there exists a birth time assignment function $t$ that avoids all the events of types $\dg, \ccc,\acc$. 
We choose $p=\frac{\ln(n)}{n}$, $\tau_0=1/n$, $z_0=\frac{1}{1-\tau_0}$ and $d=\frac{1}{2 e \ln(n)} 2^{n-1}$. 
Then $(z_0) ^n \sim e$ and $(1-p/k)^{k(n-2)} \sim 1/n$ and $(1+p/k)^{k(n-2)} \sim n$  uniformly for all $k\geq 1$. 
Plugging these values into local polynomials we get
\begin{align*}
	w_\dg(z_0) &= o(1/n), \\
	w_\ccc(z_0) &\sim \frac{1}{\ln(n)} \sum_{k=1}^{K} 4k 
	 \left(\frac{1}{2 e}\right)^k (1/n) e^k \\
	&= \frac{4}{n \ln (n)} \sum_{k=1}^{K} k 2^{-k} = o(1/n),\\
   	w_\acc(z_0)	&\sim \frac{4n }{2^{n-1}} 
   	\sum_{k=1}^{K} k^2 \left(\frac{1}{2 e } \right)^{k-1} e^k \\
   	&= \frac{4en }{2^{n-1}} \sum_{k=1}^{K} k^2 2^{-k+1} = o(1/n).
\end{align*}
Therefore $w(\frac{1}{1-\tau_0})  \leq  \tau_0$ if $n$ is large enough.
It proves that for every large enough $n$ every simple $n$-uniform hypergraph of maximum vertex degree at most $\frac{1}{2 e \ln(n)} 2^{n-1}$ is two colorable. 
Hence $D^*(n) = \Omega(\frac{2^{n}}{\ln(n)})$.
\end{proof}

The bounds on values $w_\dg(z_0)$ and $w_\acc(z_0)$ are exponentially small in terms of $n$ so the value of $w_\ccc(z_0)$ which is $\theta(\frac{1}{n\ln(n)})$ is critical for the derived bound on $D^*(n)$. 
The first element of the sum (\ref{eq:wcc})  corresponds to complete chains of length 1, i.e. the edges that are entirely contained in $B$ or in $R$. 
To avoid such events we need $p=\Omega(\ln(n)/n)$. 
It suggests that maybe choosing $p$ larger than $\ln(n)/n$ might improve the bound. Let us consider an element of the sum (\ref{eq:wcc}) for large $k$. 
The coefficient of corresponding monomial is roughly $(pn)^{k-1} (d 2^{-n+1})^k (1-p)^n$. 
Even if we ignore the increase of $(pn)^{k-1}$ when $p$ is increased, the  decrease of the factor $(1-p)^n$ is not related to $k$. 
It means that we can not essentially improve the bound unless we can bound  the maximum length of a chain by a function which is $o(\ln(n))$.
%In Section \ref{sec:bsimple} we consider different set of evente and there  the length of a chain is bounded by $\ln(n)$. 

\section{Larger number of colors - procedure $\MGCr$}
\label{sec:simpr}
Theorem \ref{thm:maxD} easily generalizes to larger number of colors. 
%Within this section we present generalization of $\MGC$ to $r$-coloring. 
%Problem of $r$-coloring of simple hypergraphs is a special case of problems treated in next section when we consider $r$-coloring of $b$-simple hypergraphs.

We present a generalization of $\MGC$ which tries to construct a proper $r$-coloring of underlying hypergraph.
We partition unit interval into consecutive intervals $C_0, P_0, C_1, P_1, \ldots, P_{r-2},C_{r-1}, P_{r-1}$ in such a way that $|C_i|=(1-p)/r$ and $|P_i|=p/r$. 
Let  $C=\bigcup_{i=0}^{r-1} C_i$ and $P=\bigcup_{i=0}^{r-1} P_i$.
Suppose that $t$ is an injective birth time assignment function which makes no edge $r$-degenerate (i.e. each edge have at least one vertex with birth time in $C$). 
The modified algorithm first assigns color $i$ to vertices with birth times in $C_i\cup P_i$. 
Vertices from $C_i$ get color $i$ permanently. 
A vertex from $P_i$ can change its color only to $(i+1)_r$. 
As long as there exists a monochromatic edge colored with $i$ and with the last vertex in some interval $P_i$ the algorithm inspects vertices with birth times in $P$ in the order induced by birth times. 
Each time when the algorithm meets a vertex $v$ with birth time in $P_i$ which is the last vertex of an edge whose all vertices are colored with $i$, the algorithm changes color of $v$ from $i$ to $(i+1)_r$. 
We call such generalized algorithm $\MGCr$.

Extending definitions from Section \ref{sec:MGC} we say that an edge is \emph{easy} if the birth times of its vertices are not contained in any interval of the form $P_{(i-1)_r},C_i,P_i$.
%it has at least two vertices with birth times in different intervals $C_i$. 
A chain $(s_1, \ldots,  s_k)$ with corresponding sequence of vertices $(v_1, \ldots, v_{k-1})$ is \emph{$r$-alternating} if there are no easy or degenerate edges in the sequence, for each $i=1, \ldots, k-1$ vertex $v_i$ is the first vertex of edge $s_i$ and the last vertex of edge $s_{i+1}$ and vertices $(v_1, \ldots, v_{k-1})$  belong to consecutive intervals $P_{(j-1)_r}, P_{(j-2)_r}, \ldots, P_{(j-k+1)_r}$ for some $j\in [r]$.
 An alternating chain $(s_1, \ldots, s_k)$ is \emph{$r$-conflicting} if $s_1$ ends in $C$ and \emph{complete $r$-conflicting} if additionally $s_k$ starts in $C$. 
 Using these definition the proof of Proposition \ref{prop:cc2} in a straightforward way generalizes to the case of $r$-coloring. 

\begin{prop}
\label{prop:ccr}
If for injective $t : V \to [0, 1)$ there are no $r$-degenerate edges and
$\MGCr$ procedure produces a coloring which is not proper, then there exists a complete $r$-conflicting chain w.r.t. $t$.
\end{prop}

We defined chains as sequences of edges. 
However sometimes it is convenient to  consider analogously defined  sequences of subsets of vertices which are not necessarily edges. 
We call such sequences \emph{chains of sets}. 
A chain of sets is $m$-uniform if every set of the chain has $m$ elements. 
Definitions of almost disjoint cycle of sets and of being disjoint,  $r$-alternating, $r$-conflicting and complete $r$-conflicting generalize in a natural way to chains of subsets. 
Using these notions we derive the following lemma.
\begin{lemma}
\label{lem:probC}
	Let $s=(s_1, \ldots, s_k)$ be an $(m+2)$-uniform disjoint chain of sets 
or almost disjoint cycle of sets with $|s_1\cap s_k|=1$. When the birth time assignment is chosen uniformly at random then
\begin{align*}
	\prob(s \text{ is $r$-alternating}) &\leq 	r (p/r)^{k-1} \left( \frac{1+p/k}{r}\right)^{m k}, 
\\
	\prob(s \text{ is $r$-conflicting}) &\leq 	r (p/r)^{k-1} \left( \frac{1}{r}\right)^{m k}, \\
%\end{align*}
%and
%\begin{align*}
	\prob(s \text{ is complete $r$-conflicting}) &\leq 	r (p/r)^{k-1} \left( \frac{1-p/k}{r}\right)^{m k}. 
\end{align*}
\end{lemma}
\begin{proof}
Let $s=(s_1, \ldots, s_k)$ an $(m+2)$-uniform disjoint chain of sets 
or almost disjoint cycle of sets with $|s_1\cap s_k|=1$ and let $(v_1, \ldots, v_{k-1})$ be its corresponding vertex sequence . 
For $s$ to be $r$-alternating it is necessary that vertices $v_1, \ldots, v_{k-1}$ have birth times in consecutive intervals $P^r_{(j-1)_r}, P^r_{(j-2)_r}\ldots, P^r_{(j-k+1)_r}$ for some $j\in [r]$. 
This happens with probability $r (p/r)^{k-1}$. 
Moreover for $i=2, \ldots, k-1$ every vertex from $s_i$ must have birth time in the arc from $t(v_{i})$ to $t(v_{i-1})$, let $\delta_i$ be the length of that arc. 
Additionally for $s$ to be alternating $s_1$ must start in point $t(v_1)$ and end in interval $C^r_{(j)_r} \cup P^r_{(j)_r}$ and $s_k$ must end in point $t(v_{k-1})$ and start in interval $P^r_{(j-k)_r} \cup C^r_{(j-k+1)_r}$. 
Let $\delta_1, \delta_k$ be the lengths of corresponding arcs. 
The probability that vertices of $s_1, \ldots, s_k$, beside $v_1, \ldots, v_{k-1}$,  have birth times in corresponding arcs is smaller than $\prod_{i=1}^k \delta_i^{m}$. 
We know that $\sum_{i=1}^k \delta_i=(k+p)/r$. 
The product is maximized when all arcs have the same length, hence the probability that $s$ is $r$-alternating is at most
\[
	r (p/r)^{k-1} \left( \frac{1+p/k}{r}\right)^{m k}.
\]
For $s$ to be $r$-conflicting $s_1$ must end in interval $C^r_{(j)_r}$, that makes 
$\sum_{i=1}^k \delta_i=k/r$ and analogous maximization of the product gives 
\[
	\prob(s \text{ is $r$-conflicting}) \leq 	r (p/r)^{k-1} \left( \frac{1}{r}\right)^{m k}.
\]
Finally, for $s$ to be complete $r$-conflicting $s_k$ must start in interval $C^r_{
(j-k+1)_r}$. Then $\sum_{i=1}^k \delta_i=(k-p)/r$ and analogously we get
\[
	\prob(s \text{ is complete $r$-conflicting}) \leq 	r (p/r)^{k-1} \left( \frac{1-p/k}{r}\right)^{m k}. 
\]
\end{proof}

\section{Hypergraphs of arithmetic progressions}
\label{sec:vdwr}
In this section we extend the results of Section \ref{sec:simp2} to derive bounds on van der Waerden numbers $W(n,r)$.
For $W,n\in \Nat$ let $H_{W,n}=(V,E)$ be the hypergraph in which $V=[W]$ and $E$ is the set of arithmetic progressions of length $n$ contained in $V$ (i.e. $E=\{(a_0, \ldots, a_{n-1})\in [W]^n : \exists_{d\in [W]} \forall_{i\in[n-2]} a_{i+1}-a_i=d\}$). 
For every $f\in E$ the minimum difference of distinct numbers from $f$ is called \emph{the difference of $f$}.  
The maximum vertex degree of $H_{W,n}$ is at most $W$. 
Clearly $r$-colorability of $H_{W,n}$ is equivalent to $W(n,r) > W$. 
We can not directly generalize Theorem \ref{thm:maxD} since these hypergraphs are not simple. 
However they are simple enough to patch the proof of Theorem \ref{thm:maxD}. 
Almost disjoint cycles need special care since now for an almost disjoint cycle $(s_1, \ldots, s_k)$ the intersection $s_1 \cap s_k$ can be large. 
We start with bounds on the number of such cycles for which $|s_1 \cap s_k|=1 $ and number of such for which $|s_1 \cap s_k|\geq 2$ .

\begin{prop}
\label{prop:almost disjoint cycle H(W,n)1}
	In any $H_{W,n}$ with maximum vertex degree $d$, any fixed vertex is contained in at most $k^2 d (nd)^{k-2} n^4$ almost disjoint cycles $(s_1, \ldots, s_k)$ of length $k$ with $|s_1 \cap s_k| =1$.	
\end{prop}
\begin{proof}
Let us consider signatures defined for almost disjoint cycles in Proposition \ref{prop:adisjC}. 
These signatures are still meaningful in $H_{W,n}$ but they no longer determine a chain uniquely. 
The problem is that hypergraph is not simple so two vertices ($v_{(i-1)_k}$ and $v_{(i+k-2)_k}$) no longer determine the last edge $s_{(i+k-1)_k}$. 
However in any $H_{W,n}$ any two vertices belong to at most $n^2$ edges so every such signature corresponds to at most $n^2$ chains. 
Multiplying the bound  from Proposition \ref{prop:adisjC} by $n^2$ we get $k^2 d (nd)^{k-2} n^4$.
\end{proof}

\begin{prop}
\label{prop:almost disjoint cycle H(W,n)2}
	There exists a polynomial $u(k,n)$ such that in any $H_{W,n}$ with maximum vertex degree $d$, any fixed vertex is contained in at most $(dn)^{k-2} u(n,k)$ almost disjoint chains $(s_1, \ldots, s_k)$ of length $k$ with $|s_1 \cap s_k| \geq 2$.	
\end{prop}
\begin{proof}
Let us fix a vertex $v$ in $H_{W,n}$ and 
let $S_c^{(\geq 2)}$ be the set of almost disjoint cycles in which the first and the last edge have at least two  vertices in common. 
%let $S_c^{(\geq 2)}$ by the set of almost disjoint cycles $(s_1, \ldots, s_k)$ (of any length) for which $|s_1 \cap s_k| \geq 2$. 
To bound the number of cycles  $s$ from $S_c^{(\geq 2)}$ of length $k$ containing vertex $v$ we consider two cases.

\textbf{Case 1: $v\in s_1 \cup s_k$.} Suppose that $v\in s_1$ (the other case is symmetric). 
Let $a,b$ be distinct vertices from $s_1 \cap s_k$. 
We assign to $s$ the following signature:
\begin{enumerate}
	\item index of $(s_1, \ldots, s_{k-2})$ among the chains of length $k-2$ containing $v$,
	\item index of $a$ within $s_1$,
	\item index of $b$ within $s_1$,
	\item index of $s_k$ among the edges containing $a,b$,
	\item index of $v_{k-1}$ within $s_k$,
	\item index of $v_{k-2}$ within $s_{k-2}$,
	\item index of $s_{k-1}$ among the edges containing $v_{k-2}, v_{k-1}$.
\end{enumerate}
It is easy to check that different $s\in S_c^{(\geq 2)}$ of length $k$ containing $v$ in $s_1$ have different signatures. 
Only the first entry of a signature is not bounded by a polynomial of $n$ and $k$. 
Using Proposition \ref{prop:disjC} we get that the number of such chains is at most $(dn)^{k-2} u_1(n,k)$ for some polynomial $u_1(n,k)$. 

\textbf{Case 2: $v$ belongs to $s'=(s_2,\ldots, s_{k-1})$.}
Let $a,b\in s_1 \cap s_k$ be distinct vertices. We assign to $s$ the following signature:
\begin{enumerate}
	\item index of $s'$ among the chains of length $k-2$ containing $v$,
	\item index of $v_{k-1}$ within $s_{k-1}$,
	\item index of $v_1$ within $s_2$,
	\item number $r_1$ such that for $\gamma_1$ being the difference of $s_1$ we have $b=a+r_1 \gamma_1$
	\item number $r_k$ such that for $\gamma_r$ being the difference of $s_k$ we have $b=a+r_k \gamma_k$
	\item number $l_1$ such that $a=v_1+ l_1 \gamma_1$,
	\item number $l_k$ such that $a=v_{k-1}+ l_k\gamma_k$,
	\item position of $v_1$ within arithmetic progression $s_1$,
	\item position of $v_{k-1}$  within arithmetic progression $s_k$.
\end{enumerate}
First three entries of the signature uniquely determine sequence $s'$ and vertices $v_{k-1}$ and $v_1$ (recall that these vertices are numbers). 
From entries (4)-(7) we can decode differences $\gamma_1, \gamma_k$ of sequences $s_1,s_k$. 
Indeed, (4) and (5) specify proportion $\gamma_1/\gamma_k =r_k/r_1$. 
Then from (6) and (7) we derive equation $v_1+ l_1\gamma_1 = v_{k-1}+ l_k\gamma_k$ from which we can obtain $\gamma_1, \gamma_k$ since we already know their proportion. 
Then we know that $s_1$ have difference $\delta_1$ and vertex $v_1$ is in the progression on the position specified by (8), hence $s_1$ is determined. 
Analogously $s_k$ is determined by $\delta_k, v_{k-1}$ and (9). 
Once again all but the first entry of the signature are bounded by some polynomial of $n,k$. 
Hence using Proposition \ref{prop:disjC} we get that the number of such chains  is at most $(dn)^{k-2} u_2(n,k)$ for some polynomial $u_2(n,k)$. 
\end{proof}

%\begin{theorem}
%	Let $W(n,r)$ be the smallest integer such that if we partition integers $\{1, \ldots, W(n,r)\}$ into $r$ sets, then at least one of them contains an arithmetic progression of length $n$. Then,
%	\[
%		W(n,r) = \Omega\left(\frac{r^n}{\ln(n)}\right)
%	\]
%\end{theorem}
\begin{proof}[Proof of Theorem \ref{thm:Wnr}]
Let us consider hypergraph of arithmetic progressions $H_{W,n}$ and let $d$ be the maximum vertex degree of $H_{W,n}$ (note that $d\leq W$). 
By Proposition \ref{prop:ccr}, in order to prove that $H_{W,n}$ is $r$-colorable it is enough to prove  that there exists an injective birth time assignment function which makes no edge $r$-degenerate and for which there are no complete $r$-conflicting chains. 
We will show that there exists function $t:V \to [0,1)$ for which no edge is $r$-degenerate, there are no disjoint complete $r$-conflicting chains and there are no $r$-alternating almost disjoint cycles. 
Let us choose function $t$ uniformly at random. 
Just like in the proof of Theorem \ref{thm:maxD} let $\dg_f,\ccc_s,\acc_c$ be events that respectively edge $f$ is $r$-degenerate, disjoint chain $s$ is complete $r$-conflicting and that almost disjoint cycle $c$ is $r$-alternating.
Analogously we define local polynomials $w_\dg(z), w_\ccc(z), w_\acc(z)$ corresponding to these events and analyse  these polynomials separately.
\\
\textbf{Events $\dg$.}  	
	We have $\prob(\dg_f)= p^n$ so $w_\dg(z)= p^n d z^n$. 
\\	
\textbf{Events $\ccc$.} By Proposition \ref{prop:disjC} the number of chains of length $k$ containing specific vertex  is at most $dk (nd)^{k-1}$.
For $s$ being a disjoint chain of length $k$ by Lemma \ref{lem:probC} we have:
\[
	\prob(\ccc_s)\leq r (p/r)^{k-1} \left( \frac{1-p/k}{r}\right)^{k(n-2)}.
\]
Hence
\begin{align}
\label{eq:wccr}
	w_{\ccc}(z)& = \sum_{k=1}^{K} 
	r (p/r)^{k-1} \left(\frac{1-p/k}{r}\right)^{k(n-2)}
	(dk(nd)^{k-1}) z^{kn} \nonumber \\
	& = \frac{r^2}{pn}
	\sum_{k=1}^{K} k  \left(\frac{dpn}{r^{n-1}}\right)^{k} (1-p/k)^{k(n-2)} z^{kn}.
\end{align}
\textbf{Events $\acc$.}  
We split these events into two classes. Let $S_c^{(1)}$ be the set of almost disjoint cycles in which the first and the last edge have exactly one vertex in common. 
Let $S_c^{(\geq 2)}$ be the set of almost disjoint cycles in which the first and the last edge have at least two  vertices in common. 
We treat events $\acc_s$ for $s\in S_c^{(1)}$ and $s\in S_c^{(\geq 2)}$ separately.
\\
\textbf{Case 1: $\acc_s$ for $s\in S_c^{(1)}$} 
 By Proposition \ref{prop:almost disjoint cycle H(W,n)1} there are at most $k^2 d  (nd)^{k-2} n^4$ almost disjoint cycles from $S_c^{(1)}$ of length $k$ containing any specific vertex. 
 Let $k$ be the length of $s$, then by Lemma \ref{lem:probC}
\[
	\prob(\acc_s) \leq r (p/r)^{k-1} \left(\frac{1+p/k}{r}\right)^{k(n-2)}.
\]
Therefore 
\begin{align}
	w_{\acc^{(1)}}(z)& = \sum_{k=3}^{K} 
	r (p/r)^{k-1} \left(\frac{1+p/k}{r}\right)^{k(n-2)} 
	(d^{k-1} n^{k+2} k^2) z^{kn} \nonumber \\
	&= \frac{r^2 n^3}{r^{n-1}} \sum_{k=3}^{K} k^2 \left(\frac{dpn}{r^{n-1}} \right)^{k-1} z^{kn}.
\end{align}
\\
\textbf{Case 2: $\acc_s$ for $s\in S_c^{(\geq 2)}$} 
There are at most $(dn)^{k-2} u(n,k)$ chains of length $k$ in $S_c^{(\geq 2)}$ containing fixed vertex (Proposition \ref{prop:almost disjoint cycle H(W,n)2}). 
Probability that an almost  disjoint cycle $(s_1, \ldots, s_k)$ is $r$-alternating is smaller than the probability that $(s_1, \ldots, s_{k-1})$ is $r$-alternating. 
The latter chain is disjoint, hence by Lemma \ref{lem:probC} the probability is at most 
\[
	r (p/r)^{k-2} \left( \frac{1+p/(k-1)}{r}\right)^{(n-2)(k-1)}.
\]
Every such chain of length $k$ contains no more than $nk$ distinct vertices. Hence the contribution to the local polynomial from the events of this type is at most
\begin{align*}
	w_{\acc^{(\geq 2)}}(z)& = \sum_{k=3}^{K} 
	r (p/r)^{k-2} \left(\frac{1+p/(k-1)}{r}\right)^{(k-1)(n-2)} 
	(dn)^{k-2} u(n,k) z^{kn} \nonumber \\
	&= \frac{1}{r^{n-3}} 
	\sum_{k=3}^{K} u(n,k) 
	\left(\frac{dnp}{r^{n-1}} \right)^{k-2} 
	\left(1+\frac{p}{k-1}\right)^{(k-1)(n-2)}
	z^{kn}.
\end{align*}

Evaluating local polynomials at $p=\frac{\ln(n)}{n}, \tau_0=1/n, z_0= \frac{1}{1-\tau_0}$ and $d = \frac{1}{2e\ln(n)} r^{n-1}$ we get
\begin{align*}
	w_\dg(z_0) &= o(1/n), \\
	w_\ccc(z_0) &\sim \frac{r^2}{\ln(n)} 
		\sum_{k=1}^{K} k \left(\frac{1}{2 e }\right)^k (1/n) e^k
\\
	&= \frac{r^2}{n \ln (n)} \sum_{k=1}^{K} k 2^{-k} = o(1/n),\\
   	w_{\acc^{(1)}}(z_0)	&\sim \frac{r^2 n^3 }{r^{n-1}} 
   	\sum_{r=1}^{K} k^2 \left(\frac{1}{2 e } \right)^{k-1} e^k \\
   	&= \frac{e\; r^2n^3 }{r^{n-1}} \sum_{k=1}^{K} k^2 2^{-k+1} = o(1/n), \\
   	w_{\acc^{(\geq 2 )}}(z_0)	&\sim \frac{1}{r^{n-3}} 
   	\sum_{r=1}^{K} u(n,k)  \left(\frac{1}{2 e} \right)^{k-1} e^k \\
   	&= \frac{e}{r^{n-3}} \sum_{k=1}^{K} u(n,k) 2^{-k+1} = o(1/n).
\end{align*}
Therefore for all large enough $n$ we get $w(z_0)\leq \tau_0$ hence by Lemma \ref{lm:local} all events of types $\dg, \ccc$ and $\acc$ can be simultaneously avoided. 
It implies that for all large enough $n$ graph $H_{\frac{r^{n-1}}{2e\ln(n)} ,n}$ is $r$-colorable.

\end{proof}

\section{$b$-simple hypergraphs}
\label{sec:bsimp}
%Hypergraph is called $b$-simple if intersection of every two different edges contains at most $b$ elements. 
%As a consequence in any $b$-simple hypergraph every set of vertices of size $b+1$ is contained in at most one edge. 
%Kostochka and Kumbhat \cite{KosKum2009} define $f(n,r,b)$ as the smallest possible number of edges in an $n$-uniform $b$-simple hypergraph that is not $r$-colorable. 
%For every positive $\epsilon$ and fixed $b$ and $r$ they proved that $f(n,r,b) = \Omega(\frac{r^{n(1+1/b)}}{n^\epsilon})$. 
%The main ingredient of their proof was a result on maximum edge-degree of $n$-uniform $b$-simple hypergraphs that is not $t$-colorable. 
%In order to improve their lower bound for $f(n,r,b)$ we derive a new bound on  maximum vertex degree in such graphs.
%the bound on maximum vertex degree in $n$-uniform $b$-simple hypergraphs that is not $t$-colorable.

We are going to analyse the behaviour of $\MGCr$ on $b$-simple $n$-uniform hypergraphs.
Proposition \ref{prop:disjC} is still valid in our case. 
In order to generalize Proposition \ref{prop:adisjC} to $b$-simple hypergraph we generalize the definition of a disjoint chain.
We say that permutation $\pi:[k]\to[k]$ is \emph{connected} if for every $i=1, \ldots, k$, the set $\{\pi(1), \ldots, \pi(i) \}$ consists of $i$ consecutive integers.
A chain $s=(s_1, \ldots, s_k)$ is \emph{$b$-disjoint} if there exists a connected permutation of its edges $\pi$ such that for every $i=2,\ldots,k$ edge  $s_{\pi(i)}$ contains at least $n-b$ vertices which are not contained in the \emph{previous edges} $s_{\pi(1)}, \ldots, s_{\pi(i-1)}$.
For chain $s=(s_1, \ldots, s_k)$ with corresponding vertex sequence $(v_1, \ldots, v_{k-1})$ any chain of sets $(s'_1, \ldots, s'_k)$ such that $s'_i \subset s_i$ is called a \emph{subchain} of $s$.

%Chain $s$ is \emph{almost $b$-disjoint} if $(s_1, \ldots, s_{k-1})$  and $(s_2, \ldots, s_{k})$ are $b$-disjoint but $s$ is not. 
%As a consequence in an almost $b$-disjoint cycle $(s_1, \ldots, s_k)$ we have $|s_k \cap \bigcup_{i=1}^{k-1} s_i | \geq b+1$ and $|s_1 \cap \bigcup_{i=2}^{k} s_i | \geq b+1$.

\begin{prop}
\label{prop:nbdisjC}
	For every $b\in \Nat$, there exists a polynomial $u(z)$ such that in any $b$-simple $n$-uniform hypergraph with maximum vertex degree $d$, any vertex belongs to at most $d^{k-1} n^{k+b} u(k) $ chains of length $k\geq 3$ that are not $b$-disjoint.
\end{prop}
\begin{proof}
	Let us fix vertex $v$ and let $s=(s_1, \ldots, s_{k})$ be a chain containing $v$ which is not $b$-disjoint. Suppose that $v\in s_i$ and let $\pi= (i,i-1, \ldots, 1, i+1, \ldots,k)$. 
Chain $s$ is not $b$-disjoint therefore there exists $j$ such that $s_{\pi(j)}$ contains at least $b+1$ elements of edges $s_{\pi(1)}, \ldots, s_{\pi(j-1)}$. Note that $\pi(j) \neq i$. We consider two cases:
\\
\textbf{Case 1. $\pi(j) < i$}. Let $S_j=\bigcup_{l=\pi(j)+1}^k s_l $. With every such $s$ we associate the following signature:
\begin{enumerate}
	\item index $i$ of the first edge containing $v$ within $s$,
	\item the smallest number $j$ such that $s_{\pi(j)}$ contains at least $b+1$ elements of edges $s_{\pi(1)}, \ldots, s_{\pi(j-1)}$,
	\item index of the chain $s'=(s_{\pi(j)+1}, \ldots, s_k)$ among the chains of length $k-\pi(j)$ containing $v$,
	\item index of the set of first $b+1$ elements of $s_{\pi(j)}\cap S_j$  among $(b+1)$-subsets of $S_j$, 
	\item index of the first vertex $w \in s_{\pi(j)} \cap s_{\pi(j)-1}$ among the vertices of $s_{\pi(j)}$,
	\item index of the chain $s''=(s_1, \ldots, s_{\pi(j)-1})$ among the chains of length $\pi(j)-1$ containing vertex $w$.
\end{enumerate}	
First entry determines permutation $\pi$. 
Second the lengths of chains $s'$ and $s''$. Length of $s'$ together with third entry determine $s'$. 
Since hypergaph is $b$-simple, fourth entry determine uniquely $s_{\pi(j)}$ (note that intersection  $s_{\pi(j)}\cap S_j$ contains at least $b+1$ elements because $\bigcup_{l=1}^{j-1} s_{\pi(l)} \subset  S_j$). 
The length of $s'$ is already known so  fifth and sixth entries determine chain $s'$. 
It shows that different chains $s$ for which $\pi(j) < i$ have different signatures. 
The number of different signatures of this type is smaller than $d^{k-1} n^{k+b} u_1(k)$ for some polynomial $u_1(k)$.

\textbf{Case 2.  $\pi(j) > i$} Analogous reasoning gives analogous bound $d^{k-1} n^{k+b} u_2(k)$ for some polynomial $u_2(k)$.

We take $u(k)= u_1(k)+u_2(k)$ and the proposition follows. 
\end{proof}

\begin{prop}
\label{prop:bdisjProb}
Let $s=(s_1, \ldots,s_k)$ be an $n$-uniform $b$-disjoint chain of edges with $n-b-k-2>k$.
For fixed $p\in (0,1)$, if the birth time assignment is chosen uniformly at random, then
\begin{align*}
	\prob(\text{$s$ is $r$-conflicting}) 
	& \leq r (p/r)^{k-1} \left( \frac{1}{r} \right)^{(n-b-1)k} (1+p)^{k^2},\\
	\prob(\text{$s$ is complete $r$-conflicting}) 
	& \leq r (p/r)^{k-1} \left( \frac{1-p/k}{r} \right)^{(n-b-1)k} \left( \frac{1+p}{1-p/k} \right)^{k^2}.\\
\end{align*}
\end{prop}
\begin{proof}
Let $(v_1, \ldots, v_{k-1})$ be the vertex sequence corresponding to chain $s$.
Vertices from this sequence are called \emph{boundary} vertices.
Definition of being $b$-disjoint implies that we can assign to each edge $s_i$ its subset $s'_i$ of size $n-b$ in such a way that sets $s'_1, \ldots, s'_k$ are pairwise disjoint.
We are going to transform sequence $(s'_1, \ldots, s'_k)$ into a disjoint chain of sets.
%We start with empty set $R$. 
For every $i=1,\ldots,k-1$ we perform the following steps:
\begin{enumerate}
	\item If $v_i$ belongs to some $s'_j$ for $j\notin \{i,i+1\}$, then remove $v_i$  form $s'_j$.
	\item If $v_i \notin s'_i$ then add $v_i$ to $s'_{i}$ and remove some not boundary vertex from $s'_{i}$.
	\item Analogously when $v_i \notin s'_{i+1}$.
\end{enumerate}
The resulting sequence $s''=(s''_1, \ldots, s''_{k})$ is a disjoint chain of sets.
The sequence is not necessarily uniform since step (1) can decrease the size of some set $s'_j$ by one. 
However, since there are only $k-1$ boundary vertices, each set $s'_i$ is decreased by one at most $k-1$ times.
In particular each set of $s''$ has at least $n-b-k+1$ elements. 
Therefore $s''$ has $(n-b-k+1)$-uniform disjoint subchain $s'''$.
Let $T$ be the set of vertices which belong to $s$ but not to $s'''$.
The original chain $s$ contains at least $k(n-b)+b$ vertices (edge $s_{\pi(1)}$ introduces $n$ new vertices and each next edge at least $n-b$).
Constructed $(n-b-k+1)$-uniform disjoint subchain $s'''$ contains exactly $(n-b-k)k+1$ vertices.
Therefore the size of $T$ equals at least $k^2-1+b$.
In this way with each $n$-uniform $b$-disjoint chain of edges $s$ of length $k$ we associate a pair $(s''',T)$ such that, $s'''$ is $(n-b-k+1)$-uniform disjoint subchain of $s$, and $T$ is a set of at least $k^2-1+b$ vertices of $s$ that does not belong to $s'''$.

Whenever $s$ is (complete) $r$-conflicting then $s'''$ is (complete) $r$-conflicting as well.
Suppose that $s'''$ is (complete) $r$-conflicting.
In order for $s$ to be (complete) $r$-conflicting, the birth time of each vertex $v$ of $T$ must belong to some specific interval of length at most $(1+p)/r$.
Indeed, let $v\in s_i\cap T$, and suppose that the birth times of $s'''_i$ are contained in $P_{(j-1)_r} \cup C_j \cup P_{j}$.
Then, the birth time of $v$ must belong to that interval as well, since otherwise edge $s_i$ would be easy.
Altogether we get
\begin{align*}
	\prob(\text{$s$ is $r$-conflicting}) 
	& \leq \prob(\text{$s'''$ is $r$-conflicting}) 
		\cdot \left( \frac{1+p}{r} \right)^{k^2-1+b},\\
	\prob(\text{$s$ is complete $r$-conflicting}) 
	& \leq \prob(\text{$s'''$ is complete $r$-conflicting})
	\cdot \left( \frac{1+p}{r} \right)^{k^2-1+b},
\end{align*}
and the proposition follows from Lemma \ref{lem:probC}.
\end{proof}
%A simple consequence of the definition is that any $b$-disjoint chain $(s_1, \ldots, s_k)$ has an $(n-b)$-uniform disjoint subchain. 

\begin{proof}[Proof of Theorem \ref{thm:Dbnr}] Let $(V,E)$ be a $b$-simple $n$-uniform  hypergraph with maximum vertex degree $d$. 
We will prove that if  $n$ is large enough and $d$ is appropriately bounded, then there is an injective function $t:V\to [0,1)$ that makes no edge $r$-degenerate and for which there are no complete $r$-conflicting chains. 
By Proposition \ref{prop:ccr} it implies that such hypergraph is $r$-colorable.
%Clearly every complete conflicting chain is either $b$-disjoint and has length at most $K$ or it contains a  prefix of length $K$ which is a conflicting $b$-disjoint chain.or a subchain which is an alternating almost $b$-disjoint cycle of length at most $K$. 
Precisely we prove that there exists an injective $t:V\to [0,1)$ and positive $K$ such that there are no degenerate edges, no complete conflicting chains of length at most $K$ and no conflicting  chains of length $K$. 
For a randomly chosen $t$ we consider following types of events
\begin{enumerate}
	\item for $f\in E$ let $\dg_f$ be the event that $f$ is $r$-degenerate,
	\item for $b$-disjoint chain $s$ of length smaller than $K$ let $\dcc_s$ be the event that $s$ is complete $r$-conflicting,
	\item for chain $s$ of length smaller than  $K$ which is not $b$-disjoint let $\ncc_s$ be the event that $s$ is complete $r$-conflicting,
	\item for $b$-disjoint chain $s$ of length $K$ let $\idcc_s$ be the event that $s$ is $r$-conflicting,
	\item for chain $s$ of length $K$ which is not $b$-disjoint let $\incc_s$ be the event that $s$ is  $r$-conflicting.
\end{enumerate}
Clearly avoiding these events is sufficient to avoid complete conflicting chains.
We are going to apply Lemma \ref{lm:local} to prove that all these events are avoided with positive probability when birth time assignment function is chosen uniformly at random.
Just like in the previous proofs we analyse separately local polynomials corresponding to these events. 
We choose parameters $p=\ln(n)/n$, $\tau_0= 1/n$ and $K=\lfloor \ln(n) \rfloor$. 
For notational convenience we also put $z_0= \frac{1}{1-\tau_0}$.
\\
\textbf{Events $\dg$.}  	
	We have $\prob(\dg_f)= p^n$ so $w_\dg(z)= p^n d z^n$ and $w_\dg(z_0)\sim (\ln(n)/n)^n d e$.
\\
\textbf{Events $\dcc$.} By Proposition \ref{prop:disjC}, the number of chains of length $k < K$ is at most $dk (nd)^{k-1}$.
%Let $s$ be $b$-disjoint chain of length $k< K$.
%Chain $s$ is $b$-disjoint therefore it has $(n-b)$-uniform subchain $s'$. 
%If $s$ is complete conflicting then also $s'$ is complete conflicting. 
Hence by Proposition \ref{prop:bdisjProb} we get
\[
	\prob(\dcc_s) \leq r (p/r)^{k-1} \left(\frac{1-p/k}{r}\right)^{(n-b-1)k} \left(\frac{1+p}{1-p/k}\right)^{k^2}
\]
and, since $\left(\frac{1+p}{1-p/k}\right)^{k^2}<6$ , the contribution of this type of events to the local polynomial is at most
\begin{align*}
	w_\dcc(z)&= 6\sum_{k=1}^{K-1} 
	r (p/r)^{k-1} \left(\frac{1-p/k}{r}\right)^{(n-b-1)k}
	dk (nd)^{k-1}
	z^{nk}	 \\
	&= 6 r^2 \sum_{k=1}^{K-1} 
	k (pn)^{k-1}  
	\left(\frac{d}{r^{n-b}}\right)^k
	(1-p/k)^{(n-b-1)k} z^{nk}.
\end{align*}
Plugging in values $p$ and $z_0$ we get
\[
	w_\dcc(z_0)\sim \frac{6r^2}{n \ln(n)} \sum_{k=1}^{K-1} 
	k 
	\left(\frac{d e \ln(n)}{r^{n-b}}\right)^k.
\]
\\
\textbf{Events $\ncc$.}
Let $s=(s_1, \ldots, s_k)$ be a chain of length $k < K$ which is not $b$-disjoint and let $(v_1, \ldots, v_{k-1})$ be the corresponding sequence of vertices.
Since the hypergraph is $b$-simple every edge $s_i$ contains at least $n-(K-1)b$ vertices which does not belong to other edges of the chain. 
Therefore there exists  a disjoint $(n-Kb)$-uniform subchain $s'$ of $s$. Lemma \ref{lem:probC} gives
\[
	\prob(\ncc_s) \leq r (p/r)^{k-1} (1/r)^{(n-Kb)k}.
\]
By Proposition \ref{prop:nbdisjC}, for some polynomial $u(k)$,  each vertex belongs to at most $d^{k-1}n^{k+b} u(k)$ chains of length $k$ which are not $b$-disjoint. 
Therefore the contribution to the local polynomial from this type of events does not exceed
\begin{align*}
	w_\ncc(z) &= \sum_{k=3}^{K-1}  
	r (p/r)^{k-1} (1/r)^{(n-Kb)k}
	d^{k-1}n^{k+b} u(k)
	z^{nk} \\
	&=
	\frac{r^2 n^{b+1}}{d} 
	\sum_{k=3}^{K-1}  
	u(k) (pn)^{k-1}
	\left( \frac{d}{r^{n-b}} \right)^k
	(1/r)^{(b(1-K)+1)k}
	z^{nk}.
\end{align*}	
For the chosen values of $p$ and $z_0$ we get
\begin{align*}
	w_\ncc(z_0) &\sim 
	\frac{r^2 n^{b+1}}{d \ln(n)} 
	\sum_{k=3}^{K-1}  
	u(k) 
	\left( \frac{d e \ln(n)}{r^{n-b}} \right)^k
	(1/r)^{(b(1-K)+1)k} \\
	&\leq 
	\frac{r^{(K-1) (b K-2)} n^{b+1}}{d \ln(n)} 
	\sum_{k=3}^{K-1}  
	u(k) 
	\left( \frac{d e \ln(n)}{r^{n-b}} \right)^k.
\end{align*}	
\\
\textbf{Events $\idcc$.}
Let $s=(s_1, \ldots, s_k)$ be a $b$-disjoint chain of length $K$. 
Considerations analogous to the case of events $\dcc$ but using the bound from Proposition \ref{prop:bdisjProb} for conflicting chains gives
\[
	\prob(\idcc_s)\leq r (p/r)^{K-1} (1/r)^{(n-b-1)K}(1+p)^{K^2}.
\]
Together with bounds from Proposition \ref{prop:disjC}, and by the fact that $(1+p)^{K^2}<4$, we obtain that the contribution to the local polynomial is at most
\begin{align*}
	w_\idcc(z) &= 4 r (p/r)^{K-1} (1/r)^{(n-b-1)K} 
				dK (nd)^{K-1}
				z^{nK} \\
			& =
				4 K r^2 (pn)^{K-1} 
				\left(\frac{d}{r^{n-b}}\right)^{K}
				z^{nK}.
\end{align*}
Evaluating in chosen $p$ and $z_0$ we get
\begin{align*}
	w_\idcc(z_0) &\sim	4\frac{K r^2}{\ln(n)} 
				\left(\frac{d e \ln(n)}{r^{n-b}}\right)^{K}.
\end{align*}
\\
\textbf{Events $\incc$.}
In the analysis of the events $\ncc$ we used upper bound for probability that is valid also for conflicting chains. Hence the same development gives
\begin{align*}
	w_\incc(z) &= 
	\frac{r^2 n^{b+1}}{d} 
	u(K) (pn)^{K-1}
	\left( \frac{d}{r^{n-b}} \right)^K
	(1/r)^{(b(1-K)+1)K}
	z^{nK}.
\end{align*}	
For the chosen values of $p$ and $z_0$ we get
\begin{align*}
	w_\incc(z_0) &\sim 
	\frac{r^{(K-1) (b K-2)} n^{b+1}}{d \ln(n)} 
	u(K) 
	\left( \frac{d e \ln(n)}{r^{n-b}} \right)^K.
\end{align*}	

Finally we choose $d=\frac{r^{n-b}}{(1+\epsilon) e^2 \ln(n)}$. 
Then $w_\dg(z_0), w_\ncc(z_0)$ and $w_\incc(z_0)$ are exponentially small in terms of $n$.  
For events $\dcc$ we get $w_\dcc(z_0)= O(\frac{1}{n\ln(n)})$. 
This time the most constraining polynomial is $w_\idcc(z)$ for which we have
\[
	w_\idcc(z_0) \sim \frac{\ln(n) r^2}{\ln(n)} \left(\frac{1}{(1+\epsilon)e}\right)^{\ln(n)} = \frac{1}{n} \frac{r^2}{(1+\epsilon)^{\ln(n)}} = o(1/n).
\]
These bounds show that for $w(z)= w_\dg(z_0)+  w_ncc(z) + w_\incc(z)+ w_\dcc(z) +w_\idcc(z)$ we have $w(\frac{1}{1-1/n})= o(1/n)$ hence by Lemma \ref{lm:local} for all large enough $n$ all bad events can be avoided. Therefore for all large enough $n$ every $b$-simple $n$-uniform hypergraph with maximum vertex degree at most $\frac{r^{n-b}}{(1+\epsilon) e^2 \ln(n)}$ is $r$-colorable. That implies 
\[
	D^b(n,r)= \Omega \left(\frac{1}{\ln(n)} r^n \right).
\]
\end{proof} 

\newcommand{\trm}{\mathcal{F}}

Let $H=(V,E)$ be a $b$-simple $n$-uniform hypergraph that is not $r$-colorable.
 Let us consider hypergraph $\trm(H)=(V,E')$, called a trimming of $H$,  whose set of edges is constructed from the original set of edges $E$ by removing from each edge a vertex of that edge with maximum degree. Let $H'= \trm^{(b)}(H)$ (we apply trimming $b$ times to $H$).
Clearly $H'$ is $(n-b)$-uniform, $b$-simple and not $r$-colorable. 
Hence if $n$ is large enough then by Theorem \ref{thm:Dbnr} hypergraph $H'$ contains a vertex $v$ of degree at least  $d =c \frac{r^n}{\ln(n)}$ for some positive constant $c$ (dependent only on $b$ and $r$).  
Vertex $v$ has at least the same degree in $H$. 
Let $F\subset E$ be the edges that contain $v$ and let $Y$ be the set of vertices that were removed from edges of $F$ during the trimming. 
Every vertex of $Y$ have degree at least $d$. 
Any $b$-subset of $Y$ together with $v$ is contained in at most one edge and every edge from $F$  contains some $b$-subset of $Y$.
Therefore for $y=|Y|$ we have $d \leq \binom{y}{b}$ which implies $y \geq d^{1/b}$. 
It means that there is at least $d^{1/b}$ vertices in $H$ of degree at least $d$. To avoid technicalities we assume that $d^{1/b}$ is integer.
Let us consider first $m=d^{1/b}$ of these vertices, we denote them by $v_1, \ldots, v_{m}$. 
Let $d_j$ be the number of edges from $E$ containing $v_j$ which contain at most $b-1$ vertices from $\{v_1, \ldots, v_{j-1}\}$. 
Clearly for every $j$ every $b$-subset of $\{v_1, \ldots, v_{j-1}\}$ determines at most one edge which contains that subset and $v_j$. 
Hence $d_j \geq d-\binom{j-1}{b}$ and therefore 
\[
	\sum_{j=1}^m d_j \geq d^{\frac{b+1}{b}}- \sum_{j=1}^m \binom{j-1}{b} 
	= d^{\frac{b+1}{b}}- \binom{m}{b+1}
	\geq d^{\frac{b+1}{b}}- \frac{m^{b+1} }{(b+1)!} 
	= d^{\frac{b+1}{b}} (1- \frac{1}{(b+1)!}).
\]
Every edge of $H$ contributes at most $b$ to the above sum so the number of edges is at least
\[
	d^{\frac{b+1}{b}} \left(1- \frac{1}{(b+1)!}\right) \frac{1}{b}.
\]
Taking into account the lower bound on $d$ we get that for fixed $b$ and $r$ the minimum number of edges in $n$-uniform $b$-simple hypergraph which is not $r$ colorable is $\Omega\left( \left(\frac{r^n}{\ln(n)} \right)^{b+1/b} \right)$ which justifies Corollary \ref{cor:fnrb}.

\bibliographystyle{siam}
\bibliography{pB}
\end{document}